\documentclass[11pt,english]{article}

\usepackage[T1]{fontenc}
\usepackage[utf8]{inputenc}
\usepackage[english]{babel}

\usepackage{a4wide}
\usepackage{authblk}
\usepackage{amsthm,amsmath,amsfonts,amssymb}
\usepackage{latexsym}
\usepackage{enumitem}

\usepackage{graphicx}
\usepackage{tikz}
\usetikzlibrary{decorations.pathreplacing,decorations.markings}
\tikzset{
  mid arrow/.style={postaction={decorate,decoration={
        markings,
        mark=at position .6 with {\arrow[#1]{stealth}}
      }}},
}

\usepackage{mathtools}
\mathtoolsset{showonlyrefs}

\def\centerarc[#1](#2)(#3:#4:#5)
    { \draw[#1] ($(#2)+({#5*cos(#3)},{#5*sin(#3)})$) arc (#3:#4:#5); }

\usepackage{times}
\usepackage{array}
\usepackage{hhline}
\usepackage{changes}
\colorlet{Changes@Color}{orange}

\numberwithin{equation}{section}

\theoremstyle{plain}
\newtheorem{theorem}{Theorem}[section]
\newtheorem{proposition}{Proposition}[section]
\newtheorem{example}{Example}[section]

\theoremstyle{definition}

\theoremstyle{remark}
\newtheorem{remark}{Remark}[section]

\def\d{\partial}

\def\dist{\hbox{dist}\,}
\def\Lip{\mathrm{Lip}}
\newcommand{\di}{\mathrm{div}\,}
\def\supp{\hbox{supp}\,}
\def\tr{\hbox{tr}\,}

\newcommand{\sgn}{\mathrm{sgn}\,}

\def\cC{\mathcal C}
\def\cD{\mathcal D}

\def\cK{\mathcal K}
\def\cH{\mathcal H}
\def\cL{\mathcal L}
\def\cN{\mathcal N}
\def\cM{\mathcal M}

\def\bR{\mathbb{R}}
\def\R{\mathbb{R}}

\newcommand{\hu}{\hat{u}}

\newcommand{\hC}{\hat{C}}

\newcommand{\dd}{\,d}

\newcommand{\lb}{\left\langle}
\newcommand{\rb}{\right\rangle}
\newcommand{\LC}{\hbox{\Large$\llcorner$}}
\def\mvint_#1{\mathchoice
          {\mathop{\vrule width 6pt height 3 pt depth -2.5pt
                  \kern -9pt \intop}\nolimits_{\kern -3pt #1}}%
          {\mathop{\vrule width 5pt height 3 pt depth -2.6pt
                  \kern -6pt \intop}\nolimits_{#1}}%
          {\mathop{\vrule width 5pt height 3 pt depth -2.6pt
                  \kern -6pt \intop}\nolimits_{#1}}%
          {\mathop{\vrule width 5pt height 3 pt depth -2.6pt
                  \kern -6pt \intop}\nolimits_{#1}}}

\begin{document}

\title{The Free Material Design  problem for stationary heat equation on low dimensional structures}

\author[{1}]{Tomasz Lewi\'nski {\footnote{t.lewinski@il.pw.edu.pl, ORCID: 0000-0003-2299-2162}}}
\author[{2}]{Piotr Rybka{\footnote{rybka@mimuw.edu.pl, ORCID:  0000-0002-0694-8201}}}
\author[{2}]{Anna Zatorska-Goldstein{\footnote{A.Zatorska-Goldstein@mimuw.edu.pl, ORCID: 0000-0002-6633-1238}}}

\affil[{1}] {{\small Faculty of Civil Engineering\\ Warsaw University of Technology\\ al. Armii Ludowej 16, 00-637 Warsaw, PL}}

\affil[{2}]
{{\small Institute of Applied
Mathematics and Mechanics\\ University of Warsaw\\ ul. Banacha 2,
02-097 Warsaw, PL}}

\date{}

\maketitle

\begin{abstract}
For a given balanced distribution of heat sources and sinks, $Q$, we find an optimal conductivity tensor field, $\hat C$, minimizing  the thermal compliance. We present $\hat C$ in a rather explicit form in term of the datum. Our solution is in a cone of non-negative  tensor valued  finite Borel measures. We present a series of examples with explicit solutions.
\end{abstract}


\section{Introduction}

\subsection{Background and an informal statement of results}
The paper concerns the problem of optimum design of local, spatially varying, anisotropic thermal properties of structural elements. The aim is to maximize the  overall heat conductivity for given thermal conditions. All the components of the conductivity tensor $C$ are viewed as design variables, while the trace of $C$ is assumed as the unit cost. The optimal distribution of conductivity components is induced by the given distribution of the heat sources within the design domain $\Omega$ and by the given heat flux applied on its boundary. Since the heat sources are prescribed as measures, possibly singular with respect to the Lebesgue measure, the optimal conductivity is expected to be a tensor valued measure.

The objective function is the so-called thermal compliance.  By performing minimization of the objective function thus chosen, we come across the optimum design setting capable of  shaping the best  material  structure by cutting off the redundant part of $\Omega$ as well as delivering the optimal distribution of the conductivity tensor field $C$ in the remaining material part of the same domain. Our main result is Theorem \ref{th:main}, which states that for a given distribution of heat sources $Q$, which is a slightly more general object than a measure, there exists an optimal conductivity tensor $\hat C$, explicitly given in terms of data. At each point of the support of measure $\hat C$  tensor $A = \frac{d \hat C}{d |\hat C|}$ has rank one. Here, $|\hat C|$ denotes the variation of $\hat C$. As a byproduct we obtain an explicit form of the  minimizer of energy $E(\hat C, \cdot)$ associated with $\hat C$, see  \eqref{df-EC}. 

The same objective function, as in the present paper, has been chosen in the books by \cite{Cherkaev2000} and \cite{Allaire2002} where a similar problem has been dealt with, yet concerning the optimal layout of two given (hence homogeneous)  isotropic materials within the design domain, the cost being the volume of one of the materials. In this setting, the correct formulation requires relaxation by homogenization.  The numerical algorithm to find the optimal conductivity can be directly constructed with using the analytical formulae of the relaxation setting (as has already been done in \cite{Goodman}) or by utilizing
some material interpolation schemes, see \cite{Donoso} and \cite{Gersborg}.  

In the present paper the problem of the optimum distribution of a  heat conducting material is considered, the cost being not the volume of a material (this volume is here unknown) but {it} 
is directly expressed by the conductivity tensor, as the integral of its trace. Our problem now is not finding an optimal layout of two distinct materials, but optimal distribution of nonhomogeneous and anisotropic conductivity properties within the design domain, admitting cutting off its part. The latter cutting off property is linked with admitting positive semi-definiteness of the conductivity tensor field to be constructed. The problem thus formulated is the scalar counterpart of the free material design problem (FMD) of creating
elastic structures of minimal compliance, subject to a single load. The FMD method has been proposed by Bendsøe et al, \cite{bendsoe}; its development has been described in \cite{Haslinger}, \cite{czarnecki} and \cite{BL}. The mass optimization by Bouchitté and Buttazzo, \cite{B-B}, has delivered the mathematical tools for the measure-theoretic setting of the FMD, cf. \cite{BL} and in particular \cite{BB}.  
Our advantage over the papers  mentioned above, which mainly deal with the vectorial problems, is relative simplicity. In  this way we may gain deeper insight into the problem.

\subsection{Statement of the problem and the results}
Throughout the paper we assume  that $\Omega\subset \bR^N$ is a bounded, open set  with a Lipschitz continuous boundary. Let us note that we do not assume convexity of $\Omega$, since due to results of \cite{B-Ch-Jimenez} (which will be discussed later) such hypothesis is  not needed.
We first present the classical setting of our problem. Later we will relax it and we will show existence of solutions to the relaxed problem. 

In the classical setting, for a given conductivity tensor $A\in L^\infty(\Omega, Sym^+(\bR^N))$, heat sources $\tilde Q$ and the flux $q$ at the boundary we consider a stationary heat conduction problem,
\begin{align} \label{classic}
 - \di A \nabla u &= \tilde{Q}, \qquad \text{in $\Omega$},\\
( A \nabla u) \cdot  \bar{n} &= q, \qquad \text{on $\d \Omega$}, \nonumber
 \end{align}
 {where $\bar n$ is the outer normal to $\Omega.$}
 
Under natural assumptions on the data, which will be discussed later, the weak form of \eqref{classic} is the Euler-Lagrange equation of the following functional,
\begin{equation*}
E_A(u) = \frac 12\int_\Omega (A\nabla u ,\nabla u ) \, dx - \int_\Omega Qu\, dx \qquad \text{for $u \in \cD(\bR^N)$},
\end{equation*}
where 
$$
\lb Q,u \rb := \int_\Omega \tilde Q(x) u(x)\,dx + \int_{\d\Omega} q(x) \gamma_u(x)\,d\cH^{N-1}(x)
$$
and $\gamma_u$ is the trace of $u$.

Then, we define the thermal compliance
by the following formula,
\begin{equation*}
    J:= \lb Q,\hu \rb \equiv\int_\Omega \tilde Q(x) \hat u(x)\,dx+ \int_{\d\Omega} q(x) \gamma_{\hat u}(x)\,d\cH^{N-1}(x),
\end{equation*}
where $\hat u$ is a solution to 
\begin{equation}\label{df-hau}
\min_u E_A(u) = E_A (\hat u)
\end{equation}
in a suitable function space {containing $\cD(\bR^N)$}.
We immediately notice that the Euler-Lagrange equation for $E_A$, i.e. the weak form of \eqref{classic}, yields,
\begin{equation*} 
J= - 2 E_A(\hat u) = 
\sup_u\left( 2 \int_\Omega Q u \, dx - \int_\Omega A\nabla u \cdot \nabla u  \, dx\right).
\end{equation*}
The advantage of this definition of $J$ is that  it does not require existence of $\hat u$. We may ask about the dependence of thermal compliance on the conductivity tensor $A$, writing $J=J(A)$, and we
can seek to find the optimal $A$ among all nonnegative-definite-matrix-valued
measures bounded by $\int_\Omega \hbox{tr}\, A\,dx \le  \Lambda_0$ (note that $\tr A$
is equivalent to the usual {operator} norm $\| A \|$ when $A$ is a nonnegative-definite
matrix). So, we consider:
\begin{equation*} 
Y = \inf \{ J(A) :\ A \in L^\infty(\Omega, Sym^+(\bR^n)), \ \int_\Omega \hbox{tr}\, A\,dx \le  \Lambda_0 \}.
\end{equation*}

At this point observe that this definition of the cost functional $Y$ should be treated very carefully. If we have a minimizing sequence $\{A_n\}_{n=1}^\infty$, then even if at each $x\in \Omega$ the matrix $A(x)$ is positive definite, then any 
limit may be only semi-definite. In addition, we have no mechanism preventing any concentration phenomena. These problems, and also the desire to reflect better real-life phenomena lead us to recast our problem using the measure theoretic tools. 

\medskip

To this end, instead of a matrix valued function $A$, we consider a matrix-valued, bounded Radon measure $C$ supported in $\overline{\Omega}$. The polar decomposition of $C$ reads  
\begin{equation}\label{df-muA}
    C = A|C|, 
\end{equation}
with $|C|$ being the total variation of the tensor measure $C$, and $A$ being a measurable, matrix-valued function $A$, such that $\| A(x)  \| = 1$ for every $x \in \overline{\Omega}$. 
Moreover, we consider a measure $Q$, with a natural constraint, $\int_{\overline{\Omega}} dQ =0$, which guarantees solvability of the Neumann problem. 
Then, the energy functional $E(C,u)$ replacing $E_A(u)$, takes the form
\begin{equation}\label{df-EC}
E(C,u) = \frac 12 \lb C, \nabla u \otimes \nabla u \rb - \lb Q, u \rb \qquad \forall u \in \cD(\bR^N).
\end{equation}

\medskip

A comment on solvability of \eqref{df-hau} in this setting is in order. In  the case $|C|$ belongs to a class of so-called multi-junction measures, $Q$ is  sufficiently regular and zero boundary data $q$ are imposed, the existence of a solution to the minimization problem \eqref{df-hau} for $E(C,u)$ in place of $E_A(u)$ was solved in \cite{rybka-azg}. Extension of this result for non-zero $q$ seems possible. Observe that for the problem we study, one does not need to know $\hat u$ in advance. However, as a byproduct of our reasoning we obtain also an existence result under much more general assumptions than these of \cite{rybka-azg} -- see Proposition \ref{existence} in Section \ref{sM}.

We may define the relaxed thermal compliance functional as 
\begin{equation}\label{J-rel}
J(C) = \sup_{u\in \cD(\bR^N)} -2 E(C,u).    
\end{equation}
The ultimate definition of the cost functional $Y$ becomes then (for a fixed $\Lambda_0 >0$),
\begin{equation}\label{df-Y}
 Y = \inf_{\hbox{tr}\,C(\overline\Omega)\le  \Lambda_0} J(C).
\end{equation}
Our main result, Theorem \ref{th:main}, consists {of} proving the existence of an optimal $C$ as well as the precise description of the solution. We illustrate the construction process by a series of examples including one based on Brothers' benchmark (see Example \ref{brothers}) .

\medskip

Let us comment briefly on the proof of Theorem \ref{th:main}. We use the minimax procedure (see Proposition \ref{min-max})) to show that
\[
Y = \sup_{u \in \cD} \, \inf_{\tr C(\overline{\Omega}) \le \Lambda_0} \left( 2 \lb Q,u \rb - \lb C, \nabla u \otimes \nabla u \rb \right). 
\]
In principle, one always has $\sup_u \inf_C E(C,u) \le \inf_C \sup_u E(C,u)$ by
an elementary argument; the point is that the equality holds. The dual problem
is much easier to work with. 

In our context, we consider $Q$ to lie in  a suitable subspace of $(\Lip_0(\overline{\Omega})/\bR)^*$ -- the dual of Lipschitz
functions modulo constants. Due to results of \cite{B-Ch-Jimenez} one can represent $Q$ equivalently as $- \di p$, where $p$ is a vector-valued measure with a certain regularity
property: $p$ is decomposed as $p = \sigma |p|$, where $\sigma(x)$ lies in a so called tangent space to the measure $|p|$ for $|p|$-a.e. $x$. The so called Kantorovich norm of $Q$ is defined as
$$
\| Q\|_1 := \sup\{ \lb Q, v \rb: v \in \Lip_1(\overline{\Omega})\};
$$

Combining the arguments of \cite{B-Ch-Jimenez} with the minimax procedure allows us to construct  an optimal Lipschitz function $\hu$  for which $\|Q\|_1$ is achieved and an optimal $\check p$ for which $\|Q\|_1=|\check p|(\overline{\Omega})$. Then we proceed to show that an optimal  $\hat C$ can be constructed precisely: its support is contained in the set $\{x: |\nabla \hu(x) | = \| \nabla \hu \|_\infty \}$ and it takes values proportional to $\nabla \hu \otimes \nabla \hu$. {Interestingly, $\check p$ equals to the heat flux corresponding to the optimal $\hat C$.}
For the precise statements we refer to the Section \ref{sM}. 

Once we established existence of an optimal tensor-field $\hC$ we address the question of solvability of (\ref{df-hau}) for this specific choice of the conductivity tensor. This is related to our previous work \cite{rybka-azg}. The result is presented in Proposition \ref{existence}.

The  whole Section \ref{examples} is devoted to presenting examples.



\subsubsection*{Notation}

We use the standard notation $Sym^+(\bR^N)$ for the space of symmetric  $N\times N$, nonnegative defined matrices. Moreover, the following function spaces will be of constant use in this note:
\begin{itemize}
\item  spaces of Lipschitz functions 
$\Lip_1(\overline{ \Omega})=\{ f:\overline{\Omega}\to\bR: \ |f(x)-f(y)|\le |x-y|,\ \forall x,y\in \overline{\Omega}\}$, and $\Lip_0(\overline{\Omega}) = \Lip (\overline{\Omega}) / \bR$;

\item $\cM(\overline{\Omega},\bR^d)$
(resp. $\cM^{+}(\overline{\Omega},\bR^d)$) - the space of $\bR^d$-valued (resp. positive) Borel measures compactly supported in $\overline{ \Omega}$. We shall abbreviate the notation to $\cM(\overline{\Omega})$ when $d=1$;

\item $\cM_b(\overline{\Omega},\bR^d)$ - the space of $\bR^d$-valued bounded Radon measures compactly supported in $\overline{ \Omega}$;

\item $\cM_b(\overline{\Omega}; Sym^+(\bR^N))$ - the cone  of nonnegative defined, symmetric  matrix-valued, bounded Radon measures supported in $\overline{\Omega}$;

\item $\cM_{\Lambda_0}(\overline{\Omega}) = \{C\in \cM_b(\overline{\Omega}; Sym^+(\bR^N)), \tr C(\overline{\Omega}) \le \Lambda_0\}$;

\item $\cM_0(\overline{ \Omega})$ - the space of signed Radon measures $\mu$ supported in $\overline{\Omega}$ such that $\int d\nu = 0$. It is endowed with the Kantorovich norm
$$
\| \nu\|_1 := \sup\{ \lb \nu, u \rb: u \in \Lip_1(\overline{\Omega})\};
$$

\item $\cM_{0,1}(\overline{ \Omega})$ - the completion of $\cM_0(\overline{\Omega})$ in $(\Lip_0(\overline{ \Omega})^*, \|\cdot \|_1)$.
\end{itemize}
 
Throughout the paper we shall also use the following notation: if $Q\in \cM(\overline{\Omega})$ and $u \in \cD(\overline{\Omega})$ or ${u\in} \Lip_1(\overline{\Omega})$ then $\lb Q, u \rb$ denotes the standard action of  $Q$ on $u$. If $Q\in \cM_{0,1}(\overline{\Omega})$ and $p\in \cM(\overline{\Omega}, \bR^N)$ is such that $\di p + Q = 0$, then for $u\in  \Lip_1(\overline{\Omega})$ the notation  $\lb p, \nabla u \rb$ shall be understood in the sense described by Proposition \ref{jimenez2}, i.e., the formula \eqref{jimenez2f}. The dot $\cdot$ denotes the scalar product of vectors in $\bR^N$.

\section{Auxiliary results} \label{aux}
We draw upon two results, which are crucial for this paper, hence we state them fully. The first one is a general version of a minimax theorem. The other one is the characterization of $\cM_{0,1}(\overline{\Omega})$.

\subsection{Minimax theorem}

\begin{proposition}[cf \cite{sorin}] \label{min-max}
Let $X$ be a compact convex subset of a topological vector space, let $Y$ be a convex set, and let $\mathcal{L}$ be a real function on the product set $X \times Y$. Assume that
\begin{enumerate}
\item[(i)] $\forall \alpha \in \R$ $\forall u_0 \in Y$ the set 
$ \{ \mu \in X \colon  \mathcal{L}(\mu, u_0) \geq \alpha\}$ is closed and convex;
\item[(ii)] $\forall \mu_0 \in X$, $\mathcal{L}(\mu_0, \cdot)$ is convex on $Y$;
\end{enumerate}
Then,
$$
\sup_{\mu \in X} \inf_{u \in Y} \mathcal{L}(\mu,u) = \inf_{u \in Y} \sup_{\mu \in X}  \mathcal{L}(\mu,u)
$$
\end{proposition}

\subsection{The characterization of $\cM_{0,1}(\overline{\Omega})$}

Here we briefly present results of \cite{B-Ch-Jimenez} which provide key tools {for} our reasoning. We start with a notion of \textit{tangent space to a measure} and \textit{tangent (vector) measure}. The former definition has been discussed in many contexts, see, e.g.,   \cite{energies,BF-JFA,B-Ch-Jimenez}. 

Let $K \subset \R^N$ be a compact set. {In this section we use the notation $C^\infty(K)$. Since we assume the boundary of $\Omega$ to be Lipschitz, it should be understood as the space of smooth functions restricted to the set $K$ (see \cite{flw} for detailed discussion of all possible definitions).} 

Let  $\mu$ be a nonnegative Borel measure compactly supported in $K$. We set
\begin{multline*}
\cN :=\Big\{ \xi \in L^\infty_\mu(K,\bR^N)\colon \exists \{u_n\}_{n=1}^\infty, u_n \in C^\infty(K),\ 
u_n \rightrightarrows0,  \ 
\nabla u_n  \to \xi\, \text{in}\, \sigma(L^\infty_\mu, L^1_\mu) \Big\}.
\end{multline*} 
Here, by writing $v_n \to \xi$, in $ \sigma(L^\infty_\mu, L^1_\mu)$ we mean convergence of $v_n$ in the weak${}^*$ topology. The orthogonal complement of $\cN$ in $L^1_\mu(K;\R^N)$, defined as
$$
\cN^\perp:= \{\eta\in L^1_\mu(K; \bR^N)\colon \int_K \eta \cdot \xi \,d\mu 
= 0 \quad \text{for all}\ \xi \in \cN\},
$$
is a closed subspace of $L^1_\mu(K;\R^N)$. The tangent space $T_\mu$ to the measure $\mu$ is then defined by the local characterization provided below.

\begin{proposition}[see \cite{B-Ch-Jimenez} Prop. 3.2]
There holds: 
\begin{enumerate}

\item There exists a $\mu$-measurable multifunction $T_\mu$ from $K$ to the subspaces of $\R^N$ such that
$$
\xi \in \cN^\perp \iff \xi(x) \in T_\mu(x) \quad \mu-a.e. x \in \R^N.
$$
\item The linear operator $u \in C^1(K) \mapsto P_\mu(x) \nabla u(x) \in L^\infty_\mu(K;\R^N)$, where $P_\mu(x)$ denotes the orthogonal projection on $T_\mu(x)$ can be extended in a unique way as a linear, continuous operator
$$
\nabla_\mu : \Lip(K) \to  L^\infty_\mu(K;\R^N),
$$
where $\Lip(K)$ is equipped with the uniform convergence on bounded subsets of $\Lip(K)$, and $L^\infty_\mu(K;\R^N)$ with the weak-$\ast$ topology.
\end{enumerate}
\end{proposition}

We can introduce the space of tangent vector  measures:
\begin{equation}\label{defMT}
\cM_T(\overline{\Omega}, \R^N) := \{ \lambda = \sigma \mu \colon \mu \in \cM^{+}(\overline{\Omega}), \quad \sigma(x) \in T_\mu(x), \,\, \mu - a.e. \}.
\end{equation}

An important ingredient of our reasoning is the following characterization of a space $\cM_{0,1}(\overline{\Omega})$.
\begin{proposition}[see \cite{B-Ch-Jimenez} Thm. 3.5 and 3.6] \label{jimenez2}
Let $\lambda \in \cM(\overline{\Omega}, \R^N)$, then $-\di \lambda \in \cM_{0,1}(\overline{\Omega})$ iff $\lambda \in \cM_T(\overline{\Omega}, \R^N)$. In this case, writing $\lambda = \sigma \mu$ as in \eqref{defMT}, we have for every $u \in \Lip(\overline{\Omega})$:
\begin{equation} \label{jimenez2f}
\lb -\di \lambda, u \rb = \int_{\overline{\Omega}}\sigma \cdot \nabla_\mu u \, d\mu =: \lb \lambda, \nabla u \rb.
\end{equation}
Moreover, the following equality holds between subsets of $\cD'(\overline{\Omega})$:
$$
\{ T_\nu \colon \nu \in \cM_{0,1}(\overline{\Omega}) \} = \{ - \di \lambda \colon \lambda \in \cM_T(\overline{\Omega}, \R^N) \}.
$$
Furthermore, for any $f \in \cM_{0,1}(\overline{\Omega})$, there exists $\overline{\lambda} \in  \cM_T(\R^N, \R^N) $ such that
$$
\| f \|_1 = |\overline{\lambda}|(\overline{\Omega}) = \min_{\lambda \in \cM_T(\overline{\Omega}, \R^N)} \left\{ \int  d |\lambda| \colon -\di \lambda = f \right\}.
$$
\end{proposition}

\begin{remark}
We stress that there are many ways to represent a measure as $\di \sigma$, due to a non-trivial kernel of the operator $\di$.
\end{remark}

\section{Main results}\label{sM}

Let $Q \in \cM_{0,1}(\overline{ \Omega}) $ be fixed. We will denote  
\begin{equation*}
\Sigma_Q = \{ p\in \cM_T(\overline{\Omega};\R^N) \colon \di p + Q = 0\}.
\end{equation*}

The thermal compliance functional $J$ is defined by \eqref{J-rel}. Its optimal value $Y$ (for a fixed $\Lambda_0 >0$) is defined by \eqref{df-Y}.
The two Propositions below, \ref{Y-Q} and \ref{Y-p} give characterization of $Y$
They provide details which we use to construct the optimal tensor $C$ -- see Theorem \ref{th:main}.

\begin{proposition} \label{Y-Q}Assume $Q \in \cM_{0,1}(\overline{\Omega})$, $\Lambda_0 >0$ is fixed, and $Y$ is the optimal value defined by \eqref{df-Y}. We have 
\begin{equation} \label{Z2}
\sqrt{Y} = \frac{1}{\sqrt{\Lambda_0}} \sup_{u\in \cD(\bR^N)\cap \Lip_1(\overline{\Omega})} \lb Q,u \rb = \frac{1}{\sqrt{\Lambda_0}} \|Q\|_1.
\end{equation}
Moreover there exists a maximizer $\hu \in Lip_1(\overline{\Omega})$.
\end{proposition}

\begin{proof}

Due to relaxed definition of $J(C)$, the value of  $Y$ takes the form of \eqref{df-Y}, i.e.,
$$
Y = \inf_{C\in\cM_{\Lambda_0}} J(C) = \inf_{C\in\cM_{\Lambda_0}} \sup_{u \in \cD(\R^N)} \big( 2\lb Q, u \rb - \lb C, \nabla u \otimes \nabla u \rb  \big).
$$

In order to apply the minimax argument presented in  Proposition \ref{min-max} let us introduce
$$
\cL(C, u) = \langle C, \nabla u \otimes \nabla u \rangle - 2\langle Q, u\rangle.
$$
For all $\alpha\in \bR$ and $u_0 \in \cD(\bR^N)$ the set $\{C\in \cM_{\Lambda_0}: \cL(C, u_0)\ge \alpha\}$ is closed and convex. For all $C_0\in \cM_{\Lambda_0}$, the function $\cL(C_0, \cdot)$ is convex on $\cD(\bR^N)$.  Moreover, $\cD(\bR^N)$ is a convex set, while $\cM_{\Lambda_0}$ is convex and compact in  the topology of weak convergence. Application of  Proposition \ref{min-max} yields
\begin{align*}
Y &= \sup_{u \in \cD(\R^N)} \inf_{C\in \cM_{\Lambda_0}} 
\big( 2 \lb Q, u \rb - \lb C, \nabla u \otimes \nabla u \rb \big) \\
&= \sup_{u \in \cD(\R^N)} 
\big( 2 \lb Q, u \rb - \sup_{C\in \cM_{\Lambda_0}} \lb C, \nabla u \otimes \nabla u \rb \big).
\end{align*}

Let us concentrate on calculating $\sup_{C\in \cM_{\Lambda_0}} \lb C, \nabla u \otimes \nabla u \rb$ for a fixed $u \in \cD(\R^N)$. We shall use the notation introduced in \eqref{df-muA}. For a fixed $u$, for any $C\in \cM_{\Lambda_0}$ we have
$$
\lb C, \nabla u \otimes \nabla u \rb= \int_{\overline{\Omega}} A\nabla u \cdot \nabla u \, d|C|
\le  \|A\nabla u\|_{L^2(|C|)} \| \nabla u \|_{L^2(|C|)}.
$$
The equality holds if and only if 
$A\nabla u = \lambda \nabla u$ for a non-negative number  $\lambda$. If we recall that $\| A\| =1$, then we see that $\lambda\in [0,1].$ Since $\Lambda_0 \ge \tr C(\overline{\Omega})$, we see that maximization of $\lambda$ requires
\begin{equation}\label{rnla}
 \lambda = \frac 1{|C|(\overline{\Omega})}\left(\Lambda_0 - \int_{\overline{\Omega}} \sum_{\lambda_i \neq \lambda} \lambda_i \,d|C|\right).   
\end{equation}
Hence, we will make $\lambda$  maximal and equal to 1  when the other eigenvalues of $A$ are zero. This means that matrix $A$ has the following  form
$$
A =  \frac{\nabla u}{|\nabla u|} \otimes \frac{\nabla u}{|\nabla u|}.
$$
As a result we showed that
\begin{equation*}
\sup_{C\in \cM_{\Lambda_0}} \lb C, \nabla u \otimes \nabla u \rb
\le \int_{\overline{\Omega}} A\nabla u \cdot \nabla u\, d\mu =\int_{\overline{\Omega}} |\nabla u |^2\,d\mu   
\end{equation*} 
for a certain measure $\mu$ which remains unspecified, yet. 

Taking into account the constraint \eqref{rnla}, we observe that $ \mu$ is a finite measure, i.e.
$$
\mu(\overline{\Omega})=\int_{\overline{\Omega}} 1 \,d\mu = \int_{\overline{\Omega}} \hbox{tr}\, A\, d\mu =  \Lambda_0.
$$
Observe also that we always have
$$
\int_{\overline{\Omega}} |\nabla u |^2 \, d\mu \le \max_{x\in \overline{\Omega}}|\nabla u(x) |^2 
\mu(\overline{\Omega}),
$$
with the equality attained only if $|\nabla u (x)| = \max_{x\in \overline{\Omega}}|\nabla u(x) | \equiv \| \nabla u \|_{L^\infty}$ for $\mu$-a.e. $x\in \overline{\Omega}$. 
Therefore, for a fixed $u\in \cD(\bR^N)$, the supremum
$$
\sup_{C\in \cM_{\Lambda_0}} \lb C, \nabla u \otimes \nabla u \rb
$$
is achieved exactly when $C$ is supported in the set $\{ x : |\nabla u(x)| = \| \nabla u \|_\infty \}$ and has values proportional to $\nabla u \otimes \nabla u$. Thus
\begin{equation} \label{Y-max}
Y = 
\sup_{u \in \cD(\R^N)} \big( 2 \lb Q,u \rb - \Lambda_0 \| \nabla u \|^2_{L^\infty} \big).
\end{equation}
Substituting a function $tu$ with $t\in \R$ instead of $u$ in \eqref{Y-max}, we find a maximum of a second degree polynomial with respect to $t$, which leads to a conclusion that  
\begin{equation} \label{r3.5}
Y = \sup_{u\in \cD(\R^N)} \left( \frac{\lb Q,u \rb^2}{\Lambda_0 \| \nabla u \|^2_{L^\infty}} \right).
\end{equation}
As a consequence we see that it is sufficient to calculate the supremum in \eqref{r3.5} over the space $\Lip_1(\overline{ \Omega})$.
Moreover, the supremum in \eqref{Y-max} is attained. Indeed, let us suppose that $u_n\in \cD(\bR^N)$ is a maximizing sequence and $x_0 \in \Omega$ is fixed. Then, the sequence $u_n - u_n(x_0)$ is also maximizing and bounded in $C(\overline{\Omega})$. Hence, $\|\nabla u_n\| \le 1$ and the Arzela-Ascoli Theorem imply that there  is a uniformly convergent subsequence  $u_{n_k} - u_{n_k} \rightrightarrows \hat u$. Moreover, $\hat u\in \Lip_1(\overline{ \Omega})$.
\end{proof}
\begin{proposition} \label{Y-p}
Assume $Q \in \cM_{0,1}(\overline{\Omega})$, $\Lambda_0 >0$ is fixed, and $Y$ is the optimal value defined by \eqref{df-Y}. We have
\begin{equation}\label{Z1}
    \sqrt{Y} = \frac{1}{\sqrt{\Lambda_0}} 
\inf_{p\in \Sigma_Q} |p|(\Omega),
\end{equation}
and the minimizer $\check p \in \cM_T(\overline{\Omega}, \R^N)$ exists.
\end{proposition}
\begin{proof}
By \eqref{r3.5} 
we already know that 
$$
\sqrt{Y} \sqrt{\Lambda_0}= \| Q\|_1 .
$$
According to Proposition \ref{jimenez2} we note
\begin{equation}\label{eq36}
 \|Q\|_1 = \min_{\lambda \in \cM_T(\overline{\Omega}, \R^N)} \left\{{\int_{\bar\Omega}  d}|\lambda| \colon -\di \lambda = Q \right\},   
\end{equation}
so the claim follows. Also, the Proposition \ref{jimenez2} yields the existence of $\check p \in  \cM_T(\R^N, \R^N) $  {minimizing \eqref{eq36}.}
\end{proof}

\bigskip
A description of the optimal tensor $C$ is given by the following result.
\begin{theorem}
\label{th:main}
Let us suppose that $Q\in \cM_{0,1}(\overline{\Omega})$, and $\Lambda_0 >0$.
Then, 
\begin{enumerate}

\item A maximizer $\hu$ in \eqref{Z2} and a minimizer $\check{p}$ in \eqref{Z1}, with polar decomposition $\check p = \sigma \mu$, {$\mu = |\check p|$,} are related by the following condition,
$$
|\check p|(\overline{\Omega}) = \lb - \di \check p, \hat u \rb = \lb \check p, \nabla \hat u \rb  .
$$
Moreover $\sigma = \nabla_\mu \hu$ $\mu$-a.e.
\item If we define  $\hat C$  by the following formula
\begin{equation}\label{df-C}
 \hat C = \frac{\Lambda_0}{\|Q\|_1}\nabla_\mu \hu \otimes \nabla_\mu \hu\, \mu,
\end{equation}
then $\hat C$ is a solution to \eqref{df-Y}, i.e. $Y = J(\hat C)$. 
\end{enumerate}
\end{theorem}

\begin{proof}
Let $\check p$ be the optimal element provided by Proposition \ref{Y-p}. By Propositions \ref{Y-Q} and \ref{Y-p} for any $1$-Lipschitz function $u$ we have
$$
\lb Q, \hu\rb = \| Q \|_1  = |\check{p}|(\overline{\Omega}) \ge 
\int_{\overline{\Omega}} \sigma \cdot \nabla_\mu u \, d\mu\  =  \lb \check p, \nabla u\rb .
$$
Next, let $\hu$ be a $1$-Lipschitz function satisfying $\lb \nabla \hu, \check p \rb = \|Q\|_1$. Since
\[
\|Q\|_1 = \lb \nabla \check p,\hu  \rb = \int_{\overline{\Omega}} \nabla_\mu \hu \cdot \sigma \dd \mu \le \int_{\overline{\Omega}} \dd \mu = \|Q\|_1,
\]
then optimality implies that $\nabla_\mu \hu = \sigma$ (or equivalently, $\nabla_\mu \hu \cdot \sigma = 1$) $\mu$-a.e.
and part 1. follows. 

Now, we proceed to prove part 2. First note that $\tr \hC = \frac{\Lambda_0}{\|Q\|_1} |\check p|$ and hence $\tr \hC(\overline{\Omega}) = \Lambda_0$, as required. Moreover, the functional $E(\hC,u)$ takes the form 
\begin{align*}
E(\hC,u) & = 2 \lb Q,u \rb - \lb \hC, \nabla u \otimes \nabla u \rb \\
& = 2 \lb -\di \check p,u \rb - \int_{\overline{\Omega}} \nabla u \otimes \nabla u \dd \hC \\
& = 2 \int_{\overline{\Omega}} \nabla_\mu u \cdot \sigma \dd \mu - \frac{\Lambda_0}{\|Q\|_1} \int_{\overline{\Omega}} (\nabla_\mu u \cdot \sigma)^2 \dd \mu \\
& = \int_{\overline{\Omega}} \left( 2 (\nabla_\mu u \cdot \sigma) - \frac{\Lambda_0}{\|Q\|_1} (\nabla_\mu u \cdot \sigma)^2 \right) \dd \mu
 \end{align*}
 Since for all real $x$ we have $2x-\frac{\Lambda_0}{\|Q\|_1}x^2 \le \frac{\|Q\|_1}{\Lambda_0}$, then we deduce 
\[
E(\hC,u) \le \int_{\overline{\Omega}} \frac{\|Q\|_1}{\Lambda_0} \dd \mu =
 \frac{\|Q\|_1}{\Lambda_0} \int_{\overline{\Omega}}\dd |\check p| =
\frac{\|Q\|_1^2}{\Lambda_0} \quad (= Y).
\]
This already shows that $\sup_u E(\hC,u) \le Y$. 
Now, we take $t = \frac{\|Q\|_1}{\Lambda_0}$, we see that $u:=t\hu$,  yields an equality in the previous estimates, which means that $\hC$ is optimal.
\end{proof}

\begin{remark}
It is interesting to check if the optimal tensor measure $\hat C$ and measure $\mu$ satisfy the assumptions of the theory developed in \cite{rybka-azg}.
\end{remark}

Once we found the optimal $\hat C$, i.e. a solution to \eqref{df-Y}, we would like to solve the minimization problem \eqref{df-hau}, i.e.
\begin{equation}\label{opt-u}
    \inf\{ E(\hat C, u): \ u\in \cD(\bR^N)\}.
\end{equation}
Here is our observation:
\begin{proposition}\label{existence}
Let us suppose that $Q$ satisfies the assumptions of Theorem \ref{th:main} and $\hat C$ is given by formula \eqref{df-C}. If $\hu\in \Lip_1(\overline{\Omega})$ is {a maximizer of (\ref{Z2})}, 
then {$\tilde u := \frac{\|Q\|_1}{\Lambda_0} \hu$}  is a solution to \eqref{opt-u}, i.e.
$$
E(\hat C, \tilde u) =\inf\{ E(\hat C, u): \ u\in \cD(\bR^N)\}.
$$
\end{proposition}
\begin{proof}
The minimization of $E(\hat C, \cdot)$ is equivalent to maximization of $-E(\hat C, \cdot).$ {We stick $\tilde u$ into $E(\hat C, \cdot)$,}
{\begin{eqnarray*}
-2 E(\hat C, \tilde u) &= &\frac{\|Q\|_1}{\Lambda_0}\left(2 \int_{\overline{\Omega}}  \sigma \cdot \nabla_\mu \hu \,d\mu -
\int_{\overline{\Omega}}  (\sigma \cdot \nabla_\mu \hu )^2\,d\mu\right)\\
&=& \frac{\|Q\|_1}{\Lambda_0}\left( \int_{\overline{\Omega}}
[1 - (1-\sigma \cdot \nabla_\mu \hu )^2 ]\,d\mu \right) \le \|Q\|_1.
\end{eqnarray*}
The equality above holds if and only if $\sigma \cdot \nabla_\mu \hu =1$ for $\mu$-a.e. $x\in \bar\Omega$. This is exactly the case for our choice of $\tilde u$.
}
\end{proof}

A few comments are in order. 
We solved here problem \eqref{df-hau} for an optimal $\hat C$ without knowing that $\mu$ is a multijunction measure, which was the assumption underlying analysis of \cite{rybka-azg}. At the same time we established a regularity result, i.e. $\tilde u\in \Lip_1(\bar\Omega)\subset H^1_\mu$. We do not have any analogue in \cite{rybka-azg}.

{\begin{remark}
It is also  interesting to see the optimal heat flux for the optimal $\hC$. Namely, we may now calculate $\tilde p = \hC \nabla_\mu\tilde u$ and we can see that $\tilde p = \check p.$ Indeed,
$$
\tilde p = \frac{\Lambda_0}{\|Q\|_1}  \lb \nabla_\mu \hu \otimes \nabla_\mu \hu, \nabla_\mu \tilde u \rb \mu = \nabla_\mu \hu \mu = \check p.
$$
\end{remark}}

\section{Examples} \label{examples}

Here, we present  a series of examples, which are illustrations of the main Theorem \ref{th:main}. We will follow a uniform style of exposition, starting from the definition of $Q$ and setting $\Lambda_0$. {We denote by $e_i$, $i=1,2$, the unit vectors of the coordinate axes.} Then, we\\
(a) check that $Q$ is an element of {$\cM_{0,1}(\overline{\Omega})$};\\
(b) find $\hat u$ yielding $\lb Q, \hu \rb = \| Q\|_1$, see \eqref{Z2};\\
(c) find $\check p\in \Sigma_Q$, which minimizes $|p|(\Omega)$ among elements of $\Sigma_Q$, see \eqref{Z1};\\
(d) write out $\hat C$, defined in \eqref{df-C}, which minimizes $J(C)$.

{We would like to emphasize that in the characterization of $\cM_{0,1}(\bar\Omega)$ provided by Proposition \ref{jimenez2} the convexity of $\Omega$ was not mentioned. Indeed, the authors of \cite{B-Ch-Jimenez} remark that the geodesic distance may be used. We will see the consequences in the example below.}


\begin{example}\label{ncx}
Let us take any $\Lambda_0>0$ and set  $\Omega = (-1, 1)^2 \setminus \{ (x_1,x_2) \in \bR^2: | x_2| \le -\frac12 x_1\}$. We define
$Q = \delta_{A}  - \delta_B, $ where $A=(-\frac12, \frac12)$, $B= {(-\frac12, -\frac12)} $. Then,\\
(1) $$\hu (x_1,x_2)= \left\{
\begin{array}{ll}
\frac {\sqrt 2}2 (x_2-x_1)& (x_1,x_2)\in \Omega, \ x_2>0, \hbox{ and } x_2>x_1,\\
\frac {\sqrt 2}2 (x_2+x_1)& (x_1,x_2)\in \Omega, \ x_2<0 \hbox{ and } x_2<-x_1,\\
0 & \hbox{otherwise};
\end{array}
\right.
$$
(2) $$
\check p = f_1 \cH^1 \LC [A,0] +f_2  \cH^1 \LC [0,B];
$$
where 
$f_1 = \frac {\sqrt 2}2(e_2 -e_1) $, $f_2 =-\frac {\sqrt 2}2(e_2 +e_1) $;\\
(3) and 
$$
\hat C =  \Lambda_0 \frac{\sqrt 2}2 
(f_1 \otimes f_1 \cH^1 \LC [A,0] +f_2\otimes f_2 \cH^1 \LC [0,B]).
$$
\end{example}
\begin{proof}
Since $\lb Q, u \rb = u(A) - u(B)$, we deduce  
that ${Q \in} \cM_{0,1}(\bar\Omega)$. Now, $\Omega$ is not convex and the geodesic distance is defined as the infimum of lengths of curves joining $A$ and $B$. Suppose that $\gamma$ is a Lipschitz path connecting $A$ and $B$ and $u$ is any element of $\Lip(\bar \Omega)$, $B= \gamma(1)$, $A = \gamma(0)$. Then,
\begin{equation}\label{r-ncop}
    \lb Q, u\rb = u(B) - u(A) = u(\gamma(1)) - u(0) + u(0) - u(\gamma(0)) \le \dist(A, 0) + \dist(B,0) = \sqrt 2,
\end{equation}
because $\Omega$ is star-like with respect to 0 and we can see that $ \dist(A, 0)=  \frac {\sqrt 2}2=  \dist(B, 0)$. We can easily check that the function defined in part (1) above turns inequality in \eqref{r-ncop} into equality. Hence, $\|Q\|_1 = \sqrt 2$.

We will use Theorem \ref{th:main}, part 1. to find the optimal $\check p$, which must satisfy,
$$
|\check p |(\bar \Omega)| = \lb \check p, \nabla_\mu \hu\rb 
= \lb - \di \check p, \hu\rb,
$$
where $\check p = \sigma\mu$, $\sigma(x) \in T_\mu(x)$ and $|\sigma(x)|=1$ for $\mu$-a.e. $x\in \bar\Omega$. Now, we are looking for a Lipschitz curve of length $\sqrt2$ and  tangent to $\nabla \hu$, which would be a support of $\mu$. 
{The easiest choice is the sum of intervals $[A,(0,0)]\cup[(0,0),B]$. Then,} it is easy to see that $\check p$ defined in part (2) has  the desired properties, in particular  
$-\di \check p = Q$.

Now, (\ref{df-C}) implies that
(3) follows immediately. 
\end{proof}

\medskip


\begin{example} (Brothers' benchmark) \label{brothers}\\
Let $\Omega= B(0,1)\subset \bR^2$, {$\Lambda_0$ is any positive number} and $$
Q = g(x_1,x_2) \cH^1\LC \d B(0,1) ,
$$
where $g(x_1,x_2) = {-4 x_1x_2}$. Then,
\begin{equation}\label{p3r2}
\hat u(x_1,x_2) = \left\{
\begin{array}{ll}
- x_2, & x_1 \ge |x_2|, \\
{-} x_1,  & x_2 \ge |x_1|, \\
 x_2,  & -x_1 > |x_2|,\\
  x_1, & -x_2 > |x_1|.
\end{array}
\right. \ 
\end{equation}
and $\check p = \sigma \mu$ where
\begin{equation}\label{p3r1}
   \sigma 
   (x_1,x_2) = 
\left\{
\begin{array}{ll}
  -\sgn (x_1) e_2. 
  & |x_1| \ge \frac{\sqrt2}{2}, \\
{-}\sgn(x_2) e_1, 
& |x_2 | \ge \frac{\sqrt2}{2},\\
 0, & |x_1|, |x_2| < \frac{\sqrt2}{2}, 
\end{array}
\right. 
\end{equation}
and $\mu = {\rho\cL^2 \LC \Omega}$, where 
\begin{equation}\label{df-ro}
\rho(x_1,x_2) = 
\left\{
\begin{array}{ll}
4|x_1|  , &  |x_1| \ge \frac{\sqrt2}2,\\
4 |x_2|,   & |x_2| \ge \frac{\sqrt2}2,\\
 0 & |x_1|, |x_2|, < \frac{\sqrt2}2 .
\end{array}
\right.
\end{equation}
Finally
{$\hat C$ is given by \eqref{CeoP}.}
\end{example}

\begin{proof}
We first notice that $Q\in \cM_{0,1}(\bar \Omega).$ However, we proceed in a different way than we did in the proof of Theorem \ref{th:main}. We first construct the optimal $\check p$, then we will look for $\hu.$

We recall  that $p\in \Sigma_Q $ if and only if $Q + \di p =0$. If $p=\sigma \mu$, where $\mu$ is a positive Radon measure and $\sigma\in L^\infty_\mu(\bar\Omega; \bR^N)$, then the distributional  definition of $\di p $ is
$$
\lb\di p, \varphi \rb = \lb\di (\sigma\mu), \varphi \rb = -  \int_{\bar\Omega}\nabla \varphi \cdot \sigma \,d\mu.
$$
In case $p\in \cM_T(\bar\Omega)$, then the above definition may be extended to $\varphi\in \Lip(\bar\Omega)$:
$$
\lb\di (\sigma\mu), \varphi \rb = -  \int_{\bar\Omega}\nabla_\mu \varphi \cdot \sigma \,d\mu.
$$
This means that in our case,  $\sigma \mu \in \Sigma_Q$ if  and only if
$$
0 = \lb Q, u \rb + \lb \di p , u \rb =
\int_{\partial B(0,1)} g(x_1,x_2) u\, d\cH^1 
 -
\int_{ B(0,1)} \sigma\cdot  \nabla_\mu u\, d\mu.
$$
We have just stated that $\di(\sigma\mu) = - Q$. We may read the above identity in a different way by using the theory of traces of measures, see \cite{Chen-Frid} and \cite{rybka-azg} in the context of the present paper. Namely, we may write,
$$
\Sigma_Q =\{ p \in \cM(\overline{\Omega}; \R^N): \di p = 0\hbox{ in }\Omega,\ p\cdot\nu = g(x_1,x_2)\hbox{ on }\d \Omega \}, 
$$
where $p\cdot\nu$ denotes the trace of the normal component of measure $p$ on $\partial \Omega$. If we pick a candidate for an optimal solution, we should check that indeed it belongs to $\cM_T(\overline{\Omega}; \R^N)$.

It is well-known, see \cite{GRS}, that the minimization problem
\begin{equation}\label{r11}
\min\{ \int_{\overline{\Omega}} d |p|:\  p \in \cM(\overline{\Omega};\R^N): \di p = 0\hbox{ in }\Omega,\ p\cdot\nu = h(x_1,x_2)\hbox{ on }\d B(0,1)\}
\end{equation}
is equivalent to 
\begin{equation}\label{r12}
\min\{\int d |D u|:\ u \in BV(\Omega),\ \gamma u = f\},    
\end{equation}
where $D$ denotes the distributional derivative od $u$. Here, $h= \frac{\d f}{\d\tau}$  and $\tau$ a tangent vector, $\nu$ the outer normal are such that $(\nu,\tau)$ is positively oriented. In the present case 
\begin{equation}\label{deff}
    f(x_1,x_2) = x_2^2- x_1^2.
\end{equation}

Since $f$ in (\ref{r12}) is continuous, we deduce from \cite{sternberg} that  there is a unique solution to \eqref{r12}. Moreover, if $v$ is a solution to \eqref{r12}, then after writing {$\check p=  \nabla v^\perp$,
where
${}^\perp$ denotes the rotation by $-\frac\pi2$, 
we obtain} a solution to \eqref{r11}, see \cite[Theorem 2.1]{GRS}. 

The solution to \eqref{r12} (with $f$ given by \eqref{deff}) is well-known (it is the Brothers' example). It is given by the following formula, see \cite{mazon},
$$
v(x_1,x_2)= \left\{
\begin{array}{ll}
1-2x_1^2 &  |x_1| \ge \frac{\sqrt2}{2},\\
2 x_2^2 -1& |x_2| \ge \frac{\sqrt2}{2}, \\
0 & |x_1|, |x_2| < \frac{\sqrt2}{2}. 
\end{array}
\right.
$$
Thus, after computing $\nabla v^\perp$, we  {want to write $\check p$ in the  following form $\check p =\sigma \mu$, where $|\sigma| =1 $ for $\mu$-a.e. $x\in \Omega$. We find $\sigma(x_1, x_2)$ 
according to (\ref{p3r1}) and
$\mu = \rho \cL^2 \LC \Omega$, where $\rho$ is given by \eqref{df-ro}}.

Now, we have to find a maximizer $\hat u$ of $\lb Q, u\rb$. We keep in mind that $\hat u$ is
{such that 
$|\check p|(\overline{\Omega}) = \lb \check p, \nabla \hat u \rb$. We notice that \eqref{p3r1} yields that the scalar product of vectors $\sigma$ 
and $\nabla_\mu \hu$ is equal to
$$
\sigma
(x_1,x_2) \cdot \nabla_\mu \hat u (x_1,x_2) 
= \left\{
\begin{array}{ll}
{\sigma}_2 \frac{\partial \hat u}{\partial x_2}  &  |x_1| \ge \frac{\sqrt2}{2},\\
{\sigma}_1 \frac{\partial \hat u}{\partial x_1} & |x_2| \ge \frac{\sqrt2}{2}, \\
0 & |x_1|, |x_2| < \frac{\sqrt2}{2}. 
\end{array}
\right.
$$
We see that {$\check p$ is an element of $\cM_T(\overline{\Omega})$ and due to the absolute continuity of $|\check p|$ we obtain,}
$$
\|\check p \| = \int_{B(0,1)} (|{\sigma}_1|\chi_{\{|x_1| \ge \frac{\sqrt2}{2}\}} +
|{\sigma}_2| \chi_{\{|x_2| \ge \frac{\sqrt2}{2}\}} )\,d\mu = \int_{B(0,1)} \sigma \cdot \nabla_\mu \hat u \,d\mu =
\lb \check p, \nabla \hat u \rb .
$$
In particular,  $|\check p|((-\frac{\sqrt2}{2}, \frac{\sqrt2}{2})^2) =0$ and $\|\check p \|  = \frac83\sqrt2$, {hence $\|Q\|_1 = \|\check p\|= \frac83\sqrt2.$}

However, the equality above is possible if and only if 
$\frac{\partial \hat u}{\partial x_2} = -\sgn(x_2)$ for $|x_2| \ge \frac{\sqrt2}{2}$ and $\frac{\partial \hat u}{\partial x_1} = -\sgn(x_1)$ for $|x_1| \ge \frac{\sqrt2}{2}$. We determine in a similar way the values of $\hu$ in the rest of $\Omega$. 
The only restriction is that 
function $\hu$ is in $\Lip_1(\bar \Omega)$. 
In particular}, we may set  $\hu$ as in \eqref{p3r2}.
{Then, one can check that $\lb Q,\hu\rb = \| \check p\|$ that confirms that the duality gap between problems (\ref{Z2}) and \eqref{eq36} vanishes}
Finally, we have to define the optimal $\hat C$. Due to formula \eqref{df-C}, we obtain,
$$
\hat A = \left\{
\begin{array}{ll}
e_2 \otimes e_2   &  |x_1| \ge \frac{\sqrt2}{2},\\
e_1 \otimes e_1  & |x_2| \ge \frac{\sqrt2}{2}, \\
0 & |x_1|, |x_2| < \frac{\sqrt2}{2}
\end{array}
\right.
$$ 
and 
\begin{equation}\label{CeoP}
{\hat C =\frac{3\sqrt2}{16} \Lambda_0 \hat A \,\rho\cL^2 \LC \Omega}.
\end{equation}
\end{proof}

\bigskip

In the next example we have the source $Q$ supported on a set of finite one-dimensional Hausdorff measure.
\begin{example}
Let us fix any $\Lambda_0>0$ and  suppose that $\Omega = (-1,1)^2$, $D^\pm = \{(x,\pm x):\ x\in (-1,1)\}$ and 
$$
Q= \cH^1\LC D^+ - \cH^1\LC D^-.
$$
Then, 
\begin{align*}
\hat u(x_1,x_2) & = \frac12(|x_1+x_2| -|x_1- x_2|),
\\
\check p &={\sqrt 2}\, \sgn(x_1) e_2 \cL^2 \LC \cC, \quad \text{where}\ \quad 
\cC =\{(x_1,x_2):\ |x_2|\le |x_1|\}, \\
\hat C &= {\frac{\Lambda_0} 2}  e_2 \otimes e_2 \cL^2 \LC \cC .
\end{align*}
\end{example}
\begin{proof}
If $u\in \Lip_1(\Omega)$, then 
$$
\lb Q, u\rb = \sqrt2\int_{-1}^1(u(x,x) - u(x,-x))\,dx \le 2\sqrt2 \int_{-1}^1|x|\,dx = 2\sqrt2.
$$
We notice that the equality above is achieved for {$\hat u(x_1,x_2) =\frac12(|x_1+x_2| -|x_1- x_2|)$. {Hence, $\|Q\|_1 = 2\sqrt2$.}

We wish to determine $\check p,$ {we will use Theorem \ref{th:main}, part 1. for this purpose, i.e. o}ur choice of $\check p$ should be such that $| \check p |(\overline{\Omega}) = \lb \check p, \nabla \hu \rb$, where
$$
\nabla \hu = \left\{
\begin{array}{ll}
 e_2    &  x_1 + x_2>0, \ x_1-x_2 >0,\\
 e_1   &  x_1 + x_2>0, \ x_1-x_2 <0,\\
 -e_2    &  x_1 + x_2<0, \ x_1-x_2 <0,\\
-e_1    &  x_1 + x_2<0, \ x_1-x_2 >0.\\
\end{array}
\right.
$$
We also want that $Q$ be represented as $Q = -\di \check p$, i.e.
$$
0= \lb Q, \hu\rb + \lb \di \check p, \hu \rb=
\int_{D^+} \hu \,d\cH^1 - \int_{D^-} \hu \,d\cH^1 - \int_{\bar\Omega}
\sigma \cdot \nabla_\mu \hu \, d\mu.
$$
{Since $|\nabla \hu| =1 $ $\cL^2$-a.e. in $\Omega$ we have a choice of the support of the measure $\mu$. We take $\cC$ as defined above, however, we could also consider $C_1 = \{(x_1, x_2) \in \Omega : |x_2| \ge |x_1|\}.$}

On the set $\cC$ the vector field $\sigma$ must be 
{equal to 
$\nabla \hu.$ We pick a simple choice for $\mu= k\cL^2\LC \cC.$}
If we take into account the form of $\nu$, the normal vector to $\partial \cC$, then by the Gauss formula we see that 
{\begin{eqnarray*}
|\check p|(\Omega) &= &\int_{\bar\Omega}
\sigma \cdot \nabla_\mu \hu \, d\mu =
\int_{\cC} k 
\, dx_1dx_2\\
&=& \int_{D^+\cup D^-} k \sigma\cdot \nu \hu \,d\cH^1=
k\frac{\sqrt2}2 \|Q\|_1.
\end{eqnarray*}
Since $\| \check p \| = \|Q\|_1,$ we deduce that}
$k = \sqrt 2$. 
Thus, $\check p = \sqrt 2\sgn(x_1)  e_2 \cL^2 \LC \cC$, as desired.
Finally, we find $\hat C = \frac{\Lambda_0}{2}
e_2 \otimes e_2 \cL^2 \LC \cC$.}
\end{proof}

\begin{remark}
 During the presentation of the above Example, we noticed that we could choose $\cC_1$ instead of $\cC$. In this way we would obtain a different solution to the minimization problem. As a result, we see that there is no unique solution to (\ref{df-Y}).
\end{remark}

\begin{example}
Let us take any $\Lambda_0>0$ and suppose that 
$$
S = \{(x,y)= {R} 
(\cos\theta , \sin\theta):\ \theta\in (\theta_0, \theta_1)\},
$$
where $\theta_0, \theta_1\in (0,2\pi)$, $R>0$
and $\Omega$ is any open, bounded set containing  $\hbox{\rm conv}\,(\{(0,0)\} \cup \overline{ S})$. We set
$$
Q = \frac{1}{\cH^1(S)} \cH^1\LC S - \delta_0.
$$
Then
$$
\hat u(x_1,x_2) =\sqrt{x_1^2 + x_2^2}{=:r},
$$
$$
\check p = {\frac 1{(\theta_1 - \theta_0)r}} 
e_r \cL^2 \LC \cK,
$$
where 
\begin{align*}
\cK &= \{ (x_1,x_2) = r(\cos\theta, \sin\theta):\ \theta\in [\theta_0, \theta_1], r\in [0, R 
]\}
\end{align*}
and
$$
\hat C = {\frac{\Lambda_0}{\cH^1(S) 
r}} e_r \otimes e_r
\cL^2 \LC \cK.
$$
\end{example}

\begin{proof}
In the present case, for $u\in\Lip_1(\overline{\Omega})$, we have
{$$
\lb Q, u \rb = \frac{1}{\cH^1(S)}\int_S (u(r,\theta) - u(0))\, d \cH^1 
\le \frac1{R(\theta_1-\theta_0)}\int_S {\sqrt{x^2_1 +x^2_2}}\, d\cH^1 =
R.
$$}
The value of the RHS above is achieved when $\hat u(x_1,x_2) =\sqrt{x_1^2 + x_2^2}$, so we have
$$
\|Q\|_1 = {R}. 
$$
Now, we are looking for an optimal $\check p = \sigma \mu$. Since $\nabla \hat u = e_r$, where $e_r(\theta) = (\cos \theta, \sin\theta)$, {then we take $\sigma = e_r$, so that we have} 
$$
\lb\check p, \nabla \hat u\rb = |\check p |(\Omega). 
$$
We expect that 
$$
\mu = 
\lambda (r) 
\cL^2 \LC \cK,
$$
where {$\lambda\in C^1([0,R])$.}
{Above all we need $\check p \in \Sigma_Q$,
\begin{equation}\label{rnb}
|\check p|(\Omega)= \lb\check p, \nabla \hat u\rb = \frac1{\cH^1(S)} \int_S \hu \,d\cH^1,
\end{equation}
i.e. 
$$
\di \check p = 0\qquad\hbox{ in } \cK.
$$
Since $\di \check p = \frac 1r \frac \partial{\partial r}\left( r\lambda(r))\right)$, we deduce that $\lambda = \frac kr,$ where $k\in \bR$. Then, \eqref{rnb} takes the form
$$
\int_\cK \frac k r \,dx_1 \,dx_2 = \int_{\theta_0}^{\theta_1} \int_0^R k \,drd\theta = k \cH^1(S) = \|Q\|_1.
$$
Thus,  $k = \|Q\|_1/\cH^1(S) = (\theta_1 - \theta_0)^{-1}$ and $\check p$ has the form we claimed.}
The choice of $\hat C$ follows from \eqref{df-C} and the elements we have.
\end{proof}

{
\begin{remark}
 The same methods works in a bit more general setting, when $S=\{(x,y)=r(\theta) e_r:\ \theta\in[\theta_0, \theta_1]\}$, where $r\in C^1[\theta_0, \theta_1]$ is positive. We could consider $Q$ given by
 $$
 Q = g \cH^1 \LC S - a \delta_0,
 $$
 where $g = e_r \cdot \nu$ and $\nu$ is the outer normal to $\cK$ at $S$. Set $\cK$ is defined as above and $a = \int_S g\, d\cH^1.$
\end{remark}}

Now, we consider $Q$ concentrated on a disconnected one-dimensional set. 
\begin{example}
Let $\Lambda_0>0$. 
We define $I = \{0\} \times [-1,1]$ and
$$
Q = \cH^1 \LC(I +g) - \cH^1 \LC (I-g),
$$
where $g=(1,1)$. We may assume that $\Omega$ is any bounded open set with smooth boundary containing the support of $Q$. Then
$$
\hat u (x_1,x_2) = \frac{\sqrt2}2 (x_1+x_2).
$$
$$
{\check p =  g \cL^2 \LC \cC 
\equiv \frac{\sqrt{2}}2(e_1 + e_2) \mu,
\quad \text{where} \quad \cC = \hbox{conv}\, (\supp Q),\quad \mu =\sqrt2\cL^2 \LC \cC}
$$
and
$$
\hat C = \frac{\Lambda_0}{{8}} \, g\otimes g  \cL^2 \LC \cC.
$$
\end{example}

\begin{proof}
We notice that if $u\in \Lip (\overline{\Omega})$, then
we have,
$$
\lb Q, u\rb = \int_{-1}^1 (u(1,y+1) - u(-1,y-1))\,dy \le \int_{-1}^1 |(2,2)|\,dy =4\sqrt 2.
$$
The inequality becomes equality for $\hat u (x_1,x_2) = \frac{\sqrt2}2 g\cdot (x_1,x_2) \equiv \frac{\sqrt2}2 (x_1+x_2).$ Indeed,
$$
\lb Q, \hu\rb =\frac{\sqrt2}2 \int_{-1}^1 (( 2+ y) -(y -2))\,dy =4 \sqrt 2.
$$
We see that $\nabla \hat u = \frac{\sqrt2}2 g$. Now, we look for $\check p\in \Sigma_Q$ such that
$\lb \check p , \nabla \hat u \rb = | \check p |(\Omega)$. We set 
{$$
\check p = \lambda g \cL^2 \LC \cC,
$$
where $\lambda >0$ and $\cC = \hbox{conv}\, (\supp Q)$. We see that
$$
\lb \check p , \nabla \hat u \rb = \sqrt 2\lambda \cL^2(\cC) = 4\sqrt2\lambda.
$$
As a result $\lambda = 1$.}
The choice of $\hat C$ follows from \eqref{df-C} and the elements we have.
\end{proof}




\subsection*{Acknowledgement}
The authors thank Michał Miśkiewicz for stimulating discussions on the subject of this paper. 
Moreover, the authors were in part supported 
by the National Science Centre, Poland. Here are  the grant numbers,
TL: 2019/33/B/ST8/00325, entitled:
{\it Merging the optimum design problems of structural topology and of the optimal choice of material characteristics.
The theoretical foundations and numerical methods};
PR:
2017/26/M/ST1/00700;
AZ-G: 2019/33/B/ST1/00535.

\end{document}